\documentclass[12pt,a4paper,reqno]{amsart}
\usepackage{graphicx}
\usepackage[all]{xy}
\usepackage{amssymb}
\usepackage{amsmath}
\usepackage[utf8]{inputenc}
\usepackage{amsthm}

\newtheorem{proposition}{Proposition}
\newtheorem{definition}{Definition}
\newtheorem{theorem}{Theorem}

\newcommand{\C}{\mathbb{C}}

\newcommand{\A}{\mathbb{A}}
\newcommand{\M}{\mathcal{M}}

\newcommand{\Span}{\mathbf{Span}}
\newcommand{\RG}{\mathbf{RG}}
\newcommand{\Grpd}{\mathbf{Grpd}}
\newcommand{\Cat}{\mathbf{Cat}}
\newcommand{\MG}{\mathbf{MG}}
\newcommand{\PreGrpd}{\mathbf{PreGrpd}}
\newcommand{\Kite}{\mathbf{DiKite}}
\newcommand{\MKite}{\mathbf{MKite}}

\begin{document}

\title[A classification theorem for Mal'tsev categories]{The Lawvere condition and a classification theorem for Mal'tsev categories}


\author{N. Martins-Ferreira}
\address[Nelson Martins-Ferreira]{Instituto Politécnico de Leiria, Leiria, Portugal}
\thanks{ }
\email{martins.ferreira@ipleiria.pt}

\begin{abstract}
A classification theorem for three different sorts of Mal'tsev categories is proven. The theorem provides a classification for Mal'tsev category, naturally Malt'sev category, and weakly Mal'tsev category in terms of classifying classes of spans. The class of all spans characterizes naturally Mal'tsev categories. The class of relations (i.e. jointly monomorphic spans) characterizes Mal'tsev categories. The class of strong relations (i.e. jointly strongly monomorphic spans) characterizes weakly Mal'tsev categories. The result is based on the uniqueness of internal categorical structures such as internal category and internal groupoid (Lawvere condition). The uniqueness of these structures is viewed as a property on their underlying reflexive graphs, restricted to the classifying spans. The class of classifying spans is combined, via a new compatibility condition, with split squares. This is analogous to orthogonality between spans and cospans. The result is a general classifying scheme which covers the main characterizations for Mal'tsev like categories. The class of positive relations has recently been shown to characterize Goursat categories and hence it is a new example that fits in this general scheme.
\keywords{Mal'tsev category \and naturally Mal'tsev \and weakly Mal'tsev \and internal category \and internal groupoid \and multiplicative graph \and directed kite \and compatibility between a split square and a span}
\end{abstract}

\maketitle

\date{Received: date / Accepted: date}

\section{Introduction}
\label{intro}
A Mal'tsev category \cite{CLP} can be defined in several equivalent ways. One possibility is to say that it is a category in which every relation is difunctional. A naturally Mal'tsev category \cite{Johnstone} can be defined as one in which every reflexive graph is the underlying graph of a unique groupoid structure (the Lawvere condition). These two notions are well known and widely studied and yet there is still lacking a general result explicitating the  features  that are common to both cases. The influencial paper \cite{Bourn} implicitly suggests a unification in terms of the fibration of pointed objects. Indeed, and in spite of not being its main purpose,  it is shown there that a category is a naturaly Mal'tsev category if and only if its fibration of pointed objects is an addive category. Moreover, it classifies Mal'tsev categories as those for which the fibration of pointed objects is unital (see \cite{BB} for further details). In addition, there is a third new case sharing some of the common properties and similarities with the other two. The notion introduced in \cite{NMF.08}, was called weakly Mal'tsev category, and it is defined as a category with local products (pullbacks of split eplimorphisms along split epimorphisms) in which every local product injection co-span is jointly epimorphic. This definition compares with the characterization of Mal'tsev categories in terms of the fibration of pointed objects \cite{Bourn}. Indeed, saying that the fibration of pointed objects is unital is the same as saying that every local product injection cospan is jointly strongly epimorphic. Furthermore, in  \cite{ZJ.NMF.12} it is proved that a category is weakly Mal'tsev if and only if every strong relation (i.e. jointly strongly monomorphic span) if difunctional. 

 Results in \cite{NMF.VdL.14} show the similarity between the three notions of Mal'tsev  categories by dealing with the appropriate classes  of spans: \\
(a) $\mathcal{M}$ is the class of all spans --- naturally Mal'tsev\\
(b) $\mathcal{M}$ is the class of all relations --- Mal'tsev categories\\
(c) $\mathcal{M}$ is the class of all strong relations --- weakly Mal'tsev categories

The purpose of this paper is to prove a general classifying theorem which combines the results from \cite{NMF.VdL.14} and \cite{ZJ.NMF.12} with those from \cite{Bourn}. The theorem is proven in Section 3. Section 2 is dedicated to terminology and definitions.

An example of a naturally Mal'tsev category is the category of abelian groups. An example of a Mal'tsev category which is not a naturally Mal'tsev category is the category of groups. An example of a category which is a weakly Mal'tsev category but not a Mal'tsev category is the category of commutative monoids with cancellation. Weakly Mal'tsev categories can also be used to detect distributive lattices amongst all lattices. Let us denote by $\mathbf{Lat}$ the category of all lattices and let $\mathbf{DLat}$ denote the category of distributive lattices. Combining results from \cite{ZJ.NMF.12} with those from  \cite{NMF.12} we get:
\begin{proposition}
Let $I\colon{\C\to \mathbf{Lat}}$ be a full subcategory of lattices and suppose it reflects pullbacks. The following conditions are equivalent:
\begin{enumerate}
\item every strong relation in $\C$ is difunctional;
\item the functor $I$ factors through $
\mathbf{DLat}$;
\item the category $\C$ is a weakly Mal'tsev category
\end{enumerate}
\end{proposition}


The context in which a Mal'tsev category is usually defined is that of a regular category. Briefly, a category is said to be regular when it has finite limits, coequalizers, and pullback stable regular epimorphisms (see for example \cite{BB} and references therein). This gives rise to a factorization system which is important when working with a calculus of relations. For the moment we are not paying attention to the classical results that characterize Mal'tsev varieties in terms of permutability for composition of equivalence relations \cite{Mal'tsev}. That study is postponed to a future work. Thus, we have chosen to establish our main result in its most general context, which goes way beyond regular categories. Surprisingly, not even products are necessary. This may seem strange --- indeed products are often used when working with naturally Mal'tsev categories. The following result, which prescribes a method for constructing examples, shows the relevance of having a category with binary products. Nevertheless, as we will see in our main result, only pullbacks and equalizers are required.

Let $\mathbb{A}$ and $\mathbb{B}$ be two categories with finite limits and let $F\colon{\mathbb{A}\to \mathbb{B}}$ be a functor preserving finite limits. If there is a natural transformation $$p_A\colon{F(A)\times F(A) \times F(A)\to F(A)}$$ such that $$p_A(x,y,y)=x=p_A(y,y,x),$$ and:
\begin{enumerate}
\item[(i)] the functor $F$ is faithful, then $\A$ is a weakly Mal'tsev category~\cite{NMF.17};
\item[(ii)]  the functor $F$ is faithful and conservative, then $\A$ is a Mal'tsev category \cite{Pedicchio};
\item[(iii)] the functor $F$ is an isomorphism, then $\A$ is a naturally Mal'tsev category~\cite{Johnstone}
\end{enumerate}

For a concrete example take $\mathbb{A}$ to be the category of abelian groups. Then we have the natural transformation $p(x,y,z)=x-y+z$. In this case the functor $F$ is the identity functor. If we take $\mathbb{A}$ to be the category of groups, not necessarily abelian, then the transformation $p$ is no longer natural. However, if we take $\mathbb{B}$ to be the category of sets and maps and letting $F$ to be the forgetful functor, then  $p$ is a natural transformation. Another example is obtained by taking the forgetful functor from the category of preordered groups into sets, which is no longer conservative but it is still faithful (see \cite{Preord}).

As explained before, our interest is to classify the three different cases of Mal'tsev like categories in terms of classifying classes of spans. This gives rise to a general scheme for classifying other possible cases. Remarkably, we only need to assume that the class of spans contains all identity spans and it is stable under pullbacks. No further assumptions are required in order to establish the equivalent conditions of Theorem~\ref{theorem 1}.

In a category with binary products, a class $\M$ of spans is said to be stable under pullbacks if the pullback of any span $(d,c)$ in $\M$, seen as a morphism $\langle d,c\rangle$ into a product, is stable under pullback. The notion is easily adapted to the case when products are not available.

Let us briefly recall how the main result from \cite{ZJ.NMF.12} would be stated if translated into the language of the present article.

Suppose $\mathbb{C}$ is a category with pullbacks and equalizers. Let $\M$ be a class of relations (i.e. jointly monic spans) in $\C$ which contains all the identity relations and is stable under pullbacks. Then  the following conditions are equivalent:
\begin{enumerate}
\item every reflexive relation in $\M$ is an equivalence relation;
\item every reflexive relation in $\M$ is a transitive relation;
\item every relation in $\M$ is a difunctional relation;
\item The corresponding class of cospans orthogonal to $\M$ contains local product injection pairs;
\end{enumerate}

The restriction on the class of spans $\M$, requireing it to consist only of relations, was overcame in \cite{NMF.VdL.14}.  However, this generalization was obtained with a price --- the condition (4), claiming that the corresponding class of cospans orthogonal to $\M$ contains local product injection pairs, had to be removed. The result, if translated to the terminology of this paper would be as Theorem \ref{theorem 1} (see Section 3) except that its condition (10) would not be present.

It is the purpose of this paper to fill in this gap. In order to understand the statements of Theorem \ref{theorem 1} let us make some observations which will then be developed in more detail in Section 2.

Considering the diagram below, let us denote by $F_i^{\M}$, with $i=1,2,3,4$, the restriction to the class $\M$ of the obvious forgetful functors, respectively ordered  1 to 4, from pregroupoids ($\PreGrpd$) to spans ($\Span$) and from multilicative graphs ($\MG$), internal categories ($\Cat$), internal groupoids ($\Grpd$) to reflexive graphs ($\RG$). More details are given below and also in Section \ref{sec:2}. 

\begin{equation}\label{diag:F1toF4}
\xymatrix{
\Grpd(\C)\ar[d]\ar[r]^{F_4} & \RG(\C)\ar@{=}[d]\\
\Cat(\C)\ar[d]\ar[r]^{F_3} & \RG(\C)\ar@{=}[d]\\
\MG(\C)\ar[r]^{F_2} & \RG(\C)\ar@<-.5ex>[d]\\
\PreGrpd(\C)\ar[u]\ar[r]^{F_1} & \Span(\C)\ar@<-.5ex>[u] & \M\ar[l]\\
}
\end{equation}

The main result of this paper establishes that 
 the following conditions are equivalent:\\
1. The functor $F_i^{\M}$ is an isomorphism, with $i=1,2,3,4$\\
2. The functor $F_i^{\M}$ has a section, with $i=1,2,3,4$\\
3. Every split square in $\C$ is $\M$-compatible.

The notion of a split square being compatible to a span is a specialization of the notion of a cospan being orthogonal to a span. The precise statement is given in Definition \ref{def1}. 

In the presence of products and pushouts it can be stated as follows: a split square (such as the one displayed in (\ref{split square})) is said to be compatible with a span $(d,c)$ when  for every outer commutative diagram, such as the one below (with notation borrowed from diagrams $(\ref{split square})$ and $(\ref{split pulback})$)
\[
\xymatrix{A+_B C 
\ar[rrr]
\ar[drr]^{\epsilon} 
\ar[rd]|{e=[e_1,e_2]}
\ar[dd]_{\alpha=[\alpha_1,\alpha_2]} 
&&& C\times A
 \ar[dd]^{d\alpha_2\times c\alpha_1}
 \\
& E
 \ar[r]^(.35){\langle p_1,p_2\rangle}
\ar@{-->}[dl]_{u} 
& A\times_B C 
\ar@{-->}[lld]^{\theta}
\ar[ru]^{\langle \pi_2,\pi_1 \rangle} 
\ar[rd]_{\delta}
\\
D
 \ar[rrr]^{\langle d,c\rangle} &&& 
 D_0\times D_1} 
\]
if there exists $u\colon{E\to D}$ such that $ue=\alpha$ and $\langle d,c\rangle u=\delta \langle p_1,p_2\rangle$, then there exists a unique morphism $\theta$ such that $\theta \epsilon =\alpha$ and   $\langle d,c\rangle \theta=\delta$. 
Note that all solid triangles commute, so, in particular, the morphism  $\delta$ is equal to $(d\alpha_2\times c\alpha_1)\langle \pi_2,\pi_1\rangle=\langle due_2\pi_2,cue_1\pi_1\rangle$ and $\epsilon=\langle p_1,p_2\rangle e$. Compare with Definition \ref{def1}.

It will be interesting to deepen the analogy with the orthogonality between $\epsilon$ and $\langle d,c \rangle$. This, however, would take us far astray and is reserved for a future work.

For convenience let us give some more details on the diagram  of categories and functors shown before, diagram (\ref{diag:F1toF4}). Further details are postponed to the next section.
In the diagram, $\RG(\C)$ denotes the category of reflexive graphs internal to $\C$. Its objects are the diagrams in $\C$ of the shape
\begin{equation}\label{diag: reflexive graph}\xymatrix{C_1 \ar@<1ex>[r]^{d} \ar@<-1ex>[r]_{c} & C_0 \ar[l]|{e} }
\end{equation}
 with $de=1_{C_0}$ and $ce=1_{C_0}$. If needed, a reflexive graph can be represented as a five-tuple $(C_1,C_0,d,e,c)$. A span is a diagram in $\C$ of the shape
\begin{equation}\label{diag: span}\xymatrix{& D \ar[ld]_{d} \ar[rd]^{c}\\D_0 && D_1}\end{equation}
 with no conditions. It is sometimes represented as a triple $(D,d,c)$. The two objects $D_0$ and $D_1$  do not usually  need to be made explicit. The functor $\RG(\C)\to\Span(\C)$ associates to every five-tuple $(C_1,C_0,d,e,c)$ the triple $(C_1,d,c)$. This means that any  class  $\M$ of spans in $\C$, which can be seen as a full subcategory of $\Span(\C)$, induces a full subcategory of $\RG(\C)$. This category can be obtained by a pullback in the category of all categories and it will be denoted $\RG(\C,\M)$. Accordingly, when the class $\M$ is  seen as a full subcategory of $\Span(\C)$ we write  $\M\cong\Span(\C,\M)$. This gives us the functors 
 \begin{enumerate}
\item[]  $F_1^{\M}\colon{\PreGrpd(\C,\M)\to\Span(\C,\M)}$
\item[]  $F_2^{\M}\colon{\MG(\C,\M)\to\RG(\C,\M)}$
\item[]  $F_3^{\M}\colon{\Cat(\C,\M)\to\RG(\C,\M)}$
\item[]   $F_4^{\M}\colon{\Grpd(\C,\M)\to\RG(\C,\M)}$
 \end{enumerate}
In section 2 we give the details on the constructions and all the new definitions that are required for the Main Theorem, which is stated and proved in section 3. 



 
\section{Review on internal categorical structures}\label{sec:2}
 
This section has the purpose of recalling some well-know definitions, establishing notation to be used later on, to introduce a new kind of categorical structures (a kite, a directed kite and a multiplicative kite) and to define a new notion for compatibility between a split squares and a spans. This new notion  is somehow similar to the notion of orthogonality between a cospan and a span, see \cite{ZJ.NMF.12}. For long time we  have struggled to find the appropriate notion of a copan orthogonal to a span (in the sense of \cite{ZJ.NMF.12}) that would complete the last item in  Theorem \ref{theorem 1}. It turns out that the solution was not to be found as an orthogonality condition between arbitrary cospans and spans but rather only between  those cospans which appear as part of a split square. Moreover, the whole structure of a spit square has to be considered, see Definition~\ref{def1}.

Throughout this paper $\C$ will denote a category  with pullbacks and equalizers. All the structures and diagrams are internal to $\C$.

\subsection{Reflexive graphs and spans}

A reflexive graph is a diagram of the shape $(\ref{diag: reflexive graph})$ in which the condition $de=1_{C_0}=ce$ holds true. It can be represented as a five-tuple $(C_1,C_0,d,e,c)$.
A morphism between reflexive graphs, say form $(C_1,C_0,d,e,c)$ to $(C'_1,C'_0,d',e',c')$, is a pair of morphisms $f=(f_1,f_0)$, displayed as
\begin{equation}\label{diag: morphims of reflexive graphs}\xymatrix{C_1 \ar@<1ex>[r]^{d} \ar@<-1ex>[r]_{c}\ar[d]_{f_1} & C_0 \ar[l]|{e}\ar[d]^{f_0} \\
C'_1 \ar@<1ex>[r]^{d'} \ar@<-1ex>[r]_{c'} & C'_0 \ar[l]|{e'} }
\end{equation}
 and such that $d'f_1=f_0d$, $c'f_1=f_0c$ and $f_1e=e'f_0$.
The category of reflexive graphs is denoted $\RG(\C)$. A span is a diagram of the shape $(\ref{diag: span})$ with no further conditions. It is also represented as $(D,d,c)$. The category of spans is denoted $\Span(\C)$. There is an obvious functor $\RG(\C)\to\Span(\C)$ assigning the span $(C_1,d,c)$ to every reflexive graph $(C_1,C_0,d,e,c)$. Any class $\M$ of spans in $\C$ can be seen as a full subcategory $\M\to\Span(\C)$. For the sake of consistency we will write $\Span(\C,\M)$ to denote the full subcategory of $\Span(\C)$ determined by the spans in the class $\M$. Similarly we obtain $\RG(\C,\M)$ as the full subcategory of $\RG(\C)$ whose span part is in $\M$, in other words, it can be seen as a pullback  in the category of categories and functors.
\[\xymatrix{\RG(\C,\M)\ar[r]\ar[d]&\RG(\C)\ar[d]\\ \Span(\C,\M) \ar[r]&\Span(\C)}\]

\subsection{Multiplicative graphs and reflexive graphs} 
 
The category of multiplicative graphs internal to $\C$ was introduced in \cite{Janelidze} and  will be denoted as $\MG(\C)$. Its objects are the diagrams in $\C$ of the form
\begin{equation}\label{diag: multiplicative graph}\xymatrix{C_2 \ar@<2.5ex>[r]^{\pi_2}\ar@<-2.5ex>[r]_{\pi_1}\ar[r]|{m} & C_1 \ar@<1.5ex>[l]|{e_1} \ar@<-1.5ex>[l]|{e_2}\ar@<1ex>[r]^{d} \ar@<-1ex>[r]_{c} & C_0 \ar[l]|{e} }
\end{equation}
in which $(C_1,C_0,d,e,c)$ is a reflexive graph,
\begin{eqnarray}
m e_1 = &1_{C_1}& = me_2\\
d m &=& d \pi_2\\
c m &=& c \pi_1,
\end{eqnarray}
the square 
\begin{equation}
\xymatrix{C_2 \ar[r]^{\pi_2}\ar[d]_{\pi_1} & C_1 \ar[d]^{c}\\
 C_1\ar[r]^{d} & C_0}
\end{equation}
is a pullback square and the maps $e_1$, $e_2$ are uniquely determined as $e_1=\langle 1_{C_1},ed\rangle$ and $e_2=\langle ec, 1_{C_1} \rangle$. 
 
A multiplicative graph, displayed as in diagram $(\ref{diag: multiplicative graph})$ will be referred to as a six-tuple $(C_1,C_0,d,e,c,m)$. The canonical morphisms from the pullback $\pi_1,\pi_2$ as well as the induced morphisms $e_1,e_2$ into the pullback are implicit.
 
Morphisms are triples $f=(f_2,f_1,f_0)$ in which $(f_1,f_0)$ is a morphism of reflexive graphs and $f_2=f_1\times_{f_0}f_1$ is such that $f_1m=m'f_2$, $f_2e_2=e'_2f_1$ and $f_2e_1=e'_1f_1$. When convenient we refer to a morphism of multiplicative graphs as $f\colon{C\to C'}$ and it should be clear that $f=(f_2,f_1,f_0)$, $C=(C_1,C_0,d,e,c,m)$ and $C'=(C'_1,C'_0,d',e',c',m)$.  

There is an obvious forgetful functor from the category of multiplicative graphs, $\MG(\C)$, to the category of reflexive graphs, $\RG(\C)$. The functor is denoted $F_2$. A class $\M$ of spans not only  gives rise to a subcategory $\RG(\C,\M)$ but also to a subcategory $\MG(\C,\M)$. This construction  can also be seen as a pullback diagram in Cat 
\[\xymatrix{\MG(\C,\M)\ar[r]^{F_2^{\M}}\ar[d]&\RG(\C,\M)\ar[d]\\ \MG(\C) \ar[r]^{F_2}&\RG(\C)}\]
so that the functor $F_2^{\M}$ is nothing but the restriction of $F_2$ to the class $\M$.

When we say that the forgetful functor from multiplicative graphs, whose span part is from $\M$, into reflexive graphs (with span part from $\M$) has a section, we really mean that there exists a natural multiplication on every reflexive graph whose span part is from $\M$. For example, if $\M$ is the class of all relations, it claims that every reflexive relation is transitive.

\subsection{The kernel pair construction}

In the proof of the main Theorem we will need a way of transforming a span into a reflexive graph. This will be done with the use of the following general construction on a span, which will be called the kernel pair construction.

Let $(D,d,c)$ be a span. The kernel pair construction is obtained by combining the kernel pairs of the morphisms $d$ and $c$ with the pullback of its projections and induced injections as illustrated.
\[
\xymatrix@!0@=4em{
D(d,c) \ar@<.5ex>[r]^-{p_2} \ar@<-.5ex>[d]_-{p_1}
& D(c) \ar@<.5ex>[l]^-{e_2} \ar@<-.5ex>[d]_-{c_1} \ar@<0ex>[r]^{c_2}
& D \ar[d]_{c}
 \\
D(d)
 \ar@<.5ex>[r]^-{d_2} \ar@<-.5ex>[u]_-{e_1} \ar[d]_{d_1}
& D
 \ar@<.5ex>[l]^-{\Delta} \ar@<-.5ex>[u]_-{\Delta} \ar[r]^{c} \ar[d]_{d}
& D_1 \\
D \ar[r]^{d}
& D_0}
\]

When $\C$ is the category of sets and maps, we may think of an element in $D$ as an arrow whose domain and codomain are drawn from different sets. In other words an element $x\in D$ is  displayed as
\[\xymatrix{D_0\ni d(x)\ar[r]^{x} & c(x)\in D_1.}\]
In view of this interpretation, the elements in $D(d)$ are the pairs $(x,y)$, $x,y\in D$, such that $d(x)=d(y)$ and they may be pictured as
\[\xymatrix{c(x) & d(x)=d(y) \ar[l]_-{x} \ar[r]^-{y} & c(y)}\] or in a simpler form as 
\[\xymatrix{\cdot & \cdot \ar[l]_-{x} \ar[r]^-{y} & \cdot}\]
Similarly, a pair $(y,z)\in D(c)$ is pictured as
\[\xymatrix{\cdot  \ar[r]^-{y} & \cdot & \cdot \ar[l]_-{z} }\]
and it follows that the elements in $D(d,c)$ are the triples $(x,y,z)$ such that $d(x)=d(y)$ and $c(y)=c(z)$, which may be pictured as
\[\xymatrix{\cdot & \cdot \ar[l]_{x} \ar[r]^-{y} & \cdot & \cdot \ar[l]_-{z} }\]

In other words, when $\C$ is the category of sets and maps we have:
\begin{eqnarray*}
d_1(x,y)=x\\
d_2(x,y)=y\\
c_1(y,z)=y\\
c_2(y,z)=z\\
\Delta(y)=(y,y)\\
p_1(x,y,z)=(x,y)\\
p_2(x,y,z)=(y,z)\\
e_1(x,y)=(x,y,y)\\
e_2(y,z)=(y,y,z)
\end{eqnarray*}

The kernel pair construction gives rise to a functor $$K_1\colon{\Span(\C)\to\RG(\C)},$$ with
$K_1(D,d,c)=(D(d,c),D,d_1p_1,\langle\Delta,\Delta\rangle,c_2p_2)$.
See \cite{NMF.VdL.14} for further information on the kernel pair construction. 


\subsection{Stability under pullbacks}

Under the assumption that the class $\M$ is stable under pullbacks, the functor $K_1$ restricts to $$K_1^{\M}\colon{\Span(\C,\M)\to\RG(\C,\M)}.$$

An alternative way of obtaining the kernel pair construction, if in the presence of binary products, is to take the following pullback
\[\xymatrix{D(d,c)\ar[r]\ar[d]& D\ar[d]^{\langle d,c\rangle}\\D\times D\ar[r]^{d\times c} & D_0\times D_1}\] 

The requirement asking that $\M$ is pullback stable means precisely that for every span $(D,d,c)$ in $\M$ and for every two morphisms $u\colon{U\to D_0}$ and $v\colon{V\to D_1}$, the span $(B,d',c')$ obtained by taking pullbacks as shown in the following picture,
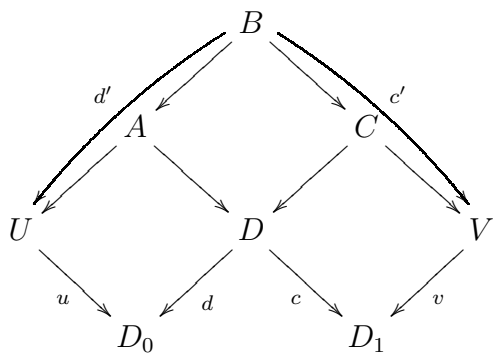
\[
\xymatrix{
& & B\ar[dl] \ar@<-.5ex>@/_/[lldd]_{d'}\ar[rd]\ar@<.5ex>@/^/[rrdd]^{c'}\\
& A\ar[ld]\ar[rd]&& C\ar[ld]\ar[rd]\\
 U\ar[rd]_{u}& & D \ar[ld]^{d}\ar[rd]_{c}&& V \ar[ld]^{v}\\
 & D_0 && D_1}\]
is still in $\M$.

\subsection{Pregroupoids}
 
 A pregroupid \cite{AndersKock}, internal to a category $\C$, consists of  a span 
\[\xymatrix{& D \ar[ld]_{d} \ar[rd]^{c}\\D0 && D1}\]
together with a pregroupoid structure. A pregroupoid structure is  a morphism $p\colon{D({d,c})\to D}$, such that 
\begin{gather}
p e_1 =d_1\quad\text{and}\quad p e_2 =c_2,\label{Mal'tsev-conditions}\\
dp =d c_2 p_2\quad\text{and}\quad cp = c d_1 p_1.\label{Domain-and-Codomain}
\end{gather}

The object $D(d,c)$ is obtained together with the maps $$d_1,d_2,c_1,c_2,p_1,p_2,e_1,e_2$$ by means of the kernel pair construction, as explained in the previous subsection. In set-theoretical terms, the object $D({d,c})$ consists on those triples $(x,y,z)$ of arrows in $D$ for which $d(x)=d(y)$ and $c(y)=c(z)$, so that the two conditions $(\ref{Mal'tsev-conditions})$ are 
\[p(x,y,y)=x,\quad p(y,y,z)=z\] while the two conditions $(\ref{Domain-and-Codomain})$ become
\[dp(x,y,z)=d(z), \quad cp(x,y,z)=c(x). \]

In this way we form the category of pregroupoids with its span part drawn from the class $\M$. It will be denoted as $\PreGrpd(\C,\M)$.


\subsection{Internal categories and internal groupoids}

An internal category is a multiplicative graph in which the multiplication is associative. The category of internal categories to $\C$ is denoted $\Cat(\C)$. A groupoid is an internal category in which every morphism is invertible. Internally,  it can be seen as an associative multiplicative graph in which the square 
\begin{equation}
\xymatrix{C_2 \ar[r]^{\pi_2}\ar[d]_{\pi_1} & C_1 \ar[d]^{d}\\
 C_1\ar[r]^{d} & C_0}
\end{equation}
is a pullback (see \cite{BB}). The category of internal groupoids internal to $\C$ is denoted $\Grpd(\C)$. This explains the list of forgetful functors $F_i$, $i=2,3,4$, and the vertical inclusions shown in Diagram \ref{diag:F1toF4}.


In a similar manner as before we define the categories $\Cat(\C,\M)$ and $\Grpd(\C,\M)$ of internal categories and internal groupoids in $\C$ with respect to a class $\M$ of spans.

\subsection{Multiplicative kites}

The notion of a kite was first considered in \cite{NMF.08} as admissibility diagram. It was then considered in \cite{NMF.VdL.14} as a kite. It's main purpose is to generalize the structure of a groupoid and a pregroupoid so that it can be used as a setting where it is possible to transform a groupoid into a pregroupoid and vice versa.

A kite, internal to $\C$, is a diagram of the form
\begin{equation}\label{kite}
\vcenter{\xymatrix@!0@=4em{A \ar@<.5ex>[r]^-{f} \ar[rd]_-{\alpha} & B
\ar@<.5ex>[l]^-{r}
\ar@<-.5ex>[r]_-{s}
\ar[d]^-{\beta} & C \ar@<-.5ex>[l]_-{g} \ar[ld]^-{\gamma}\\
& D}}
\end{equation}
with $fr=1_{B}=gs$, $\alpha r=\beta=\gamma s$.

A directed kite is a kite together with a span $(D,d,c)$ such that $d\alpha=d\beta f$, $c\beta g=c\gamma$.

 Once again, if the span part of a kite is required to be in $\M$ then it is an object in the  category $\Kite(\C,\M)$, where the morphisms are the natural transformations between such diagrams.

Each diagram such as $(\ref{kite})$ induces a diagram
\begin{equation}
\vcenter{\xymatrix@!0@=3em{ & C \ar@<.5ex>[ld]^-{e_2} \ar@<-.5ex>[rd]_-{g}
\ar@/^/[rrrd]^-{\gamma} \\
A\times_{B}C \ar@<.5ex>[ru]^-{\pi_2}
\ar@<-.5ex>[rd]_-{\pi_1} && B \ar@<.5ex>[ld]^-{r} \ar@<-.5ex>[lu]_-{s}
 \ar[rr]|-{\beta} && D\\
& A \ar@<.5ex>[ru]^-{f} \ar@<-.5ex>[lu]_-{e_1} \ar@/_/[urrr]_-{\alpha}}}\end{equation}
in which the double diamond is a double split epimorphism (or a split square). The morphisms $e_1, e_2$ are determined as $e_1=\langle 1_A, sf\rangle$ and $e_2=\langle rg, 1_C\rangle$.

A multiplication on a kite is a morphism $m\colon{A\times_{B} C\to D}$ such that $dm=d\gamma\pi_2$, $cm=c\alpha\pi_1$, $me_1=\alpha$ and $me_2=\gamma$.

We will consider the forgetful functor from the category of multiplicative kites into the category of directed kites, with direction (that is the span part) drawn from the class $\M$. This functor will be called $F_0^{\M}$ in Theorem \ref{theorem 1}.

\[\xymatrix{\MKite(\C,\M)\ar[d]^{F_0^{\M}} \\ \Kite(\C,\M)}\]

One more ingredient is needed in order to understand the statement of Theorem \ref{theorem 1}, which is the notion of a split square being compatible with a span. Before entering into that let us first give a short list of examples for a directed kite. This information will be used  in the proof of Theorem~\ref{theorem 1}.

List of examples of directed kites, for a later reference:
\begin{enumerate}
\item If $(C_1,C_0,d,e,c)$ is a reflexive graph then the following diagram is a directed kite
\begin{equation}\label{diag: kite1}
\vcenter{\xymatrix{C_1 \ar@<.5ex>[r]^-{d} \ar@{=}[rd]_-{} & C_0
\ar@<.5ex>[l]^-{e}
\ar@<-.5ex>[r]_-{e}
\ar[d]^-{e} & C_1 \ar@<-.5ex>[l]_-{c} \ar@{=}[ld]^-{}\\
& C_1\ar[dl]_{d}\ar[rd]^{c}\\\cdot&&\cdot}}
\end{equation}

This directed kite is multiplicative if and only if the reflexive graph is a multiplicative graph.

\item If $(C_1,C_0,d,e,c,m)$ is a multiplicative graph then the following diagram is a directed kite
\begin{equation}\label{diag: kite2}
\vcenter{\xymatrix{C_2 \ar@<.5ex>[r]^-{\pi_2} \ar[rd]_-{m} & C_1
\ar@<.5ex>[l]^-{e_2}
\ar@<-.5ex>[r]_-{e_1}
\ar@{=}[d]^-{} & C_1 \ar@<-.5ex>[l]_-{\pi_1} \ar[ld]^-{m}\\
& C_1\ar[dl]_{d}\ar[rd]^{c}\\\cdot&&\cdot}}
\end{equation}

This directed kite has a unique multiplicative structure if and only if the multiplicative graph is associative (i.e., an internal category).

\item If $(C_1,C_0,d,e,c,m)$ is an associative multiplicative graph (that is, an internal category) then the following diagram is a directed kite
\begin{equation}\label{diag: kite3}
\vcenter{\xymatrix{C_2 \ar@<.5ex>[r]^-{m} \ar[rd]_-{\pi_2} & C_1
\ar@<.5ex>[l]^-{e_2}
\ar@<-.5ex>[r]_-{e_1}
\ar@{=}[d]^-{} & C_1 \ar@<-.5ex>[l]_-{m} \ar[ld]^-{\pi_1}\\
& C_1\ar[dl]_{d}\ar[rd]^{c}\\\cdot&&\cdot}}
\end{equation}

This directed kite is multiplicative if and only the internal category is an internal groupoid (see \cite{NMF.08}).

\item If $(f_1,f_0)\colon{(C_1,C_0,d,e,c)\to (C'_1,C'_0,d',e',c')}$ is a morphism of reflexive graphs then the following diagram is a directed kite
\begin{equation}\label{diag: kite4}
\vcenter{\xymatrix{C_1 \ar@<.5ex>[r]^-{d} \ar@{->}[rd]_-{f_1} & C_0
\ar@<.5ex>[l]^-{e}
\ar@<-.5ex>[r]_-{e}
\ar[d]^-{e'f_0} & C_1 \ar@<-.5ex>[l]_-{c} \ar@{->}[ld]^-{f_1}\\
& C'_1\ar[dl]_{d'}\ar[rd]^{c'}\\\cdot&&\cdot}}
\end{equation}

If the morphism of reflexive graphs can be extended to a morphism of multiplicative reflexive graphs then the induced directed kite represented in the diagram above is multiplicative.

\item If $(D,d,c)$ is a span then the kernel pair construction gives a directed kite as follows
\begin{equation}\label{diag: kite5}
\vcenter{\xymatrix{D(d) \ar@<.5ex>[r]^-{d_2} \ar@{->}[rd]_-{d_1} & D
\ar@<.5ex>[l]^-{\Delta}
\ar@<-.5ex>[r]_-{\Delta}
\ar@{=}[d]^-{} & D(c) \ar@<-.5ex>[l]_-{c_1} \ar@{->}[ld]^-{c_2}\\
& D\ar[dl]_{d}\ar[rd]^{c}\\\cdot&&\cdot}}
\end{equation}

This yields a reflection between the category of directed kites  and the category of spans

\[\xymatrix{\Kite \ar@<.5ex>[r]  &\Span\ar@<.5ex>[l]}\]

A directed kite goes to its direction span, a span goes to the directed kite displayed above. Moreover, the span $(D,d,c)$ is a pregroupoid if and only if its associated directed kite is multiplicative.

\item One last example that we will need later on is obtained from a split square. Indeed, any split square as illustrated in the diagram below, display $(\ref{split square})$, gives rise to a directed kite as illustrated below.
\begin{equation}\label{diag: kite6}
\vcenter{\xymatrix@!0@=4em{A \ar@<.5ex>[r]^-{f} \ar[rd]_-{e_1} & B
\ar@<.5ex>[l]^-{r}
\ar@<-.5ex>[r]_-{s}
\ar[d]|-{e_1r=e_2s} & C \ar@<-.5ex>[l]_-{g} \ar[ld]^-{e_2}\\
& D\ar[dl]_{p_2}\ar[rd]^{p_1}\\C&&A}}
\end{equation}

\end{enumerate}

\subsection{Split squares compatible with spans}

A split square is a diagram of the shape
\begin{equation}\label{split square}
\xymatrix@!0@=4em{
E
\ar@<.5ex>[r]^-{p_2} \ar@<-.5ex>[d]_-{p_1}
& C
\ar@<.5ex>[l]^-{e_2} \ar@<-.5ex>[d]_-{g}
 \\
A
 \ar@<.5ex>[r]^-{f} \ar@<-.5ex>[u]_-{e_1}
& B
 \ar@<.5ex>[l]^-{r} \ar@<-.5ex>[u]_-{s}
 }
\end{equation}
such that $fr=1_B=gs$, $p_2e_2=1_C$, $p_1e_1=1_A$, $p_2e_1=sf$, $p_1e_2=rg$ and  $gp_2=fp_1$.

Every split square gives rise to a comparison morphism $p\colon{E\to A\times_B C}$ into the pullback of $g$ along $f$, which will be denoted as a split pullback, and it is obtained as
\begin{equation}\label{split pulback}
\xymatrix@!0@=4em{
A\times_B C
\ar@<.5ex>[r]^-{\pi_2} \ar@<-.5ex>[d]_-{\pi_1}
& C
\ar@<.5ex>[l]^-{\epsilon_2} \ar@<-.5ex>[d]_-{g}
 \\
A
 \ar@<.5ex>[r]^-{f} \ar@<-.5ex>[u]_-{\epsilon_1}
& B
 \ar@<.5ex>[l]^-{r} \ar@<-.5ex>[u]_-{s}
 }
\end{equation}
with $\epsilon_1=\langle 1_{A},sf\rangle$ and $\epsilon_2=\langle rg, 1_{C} \rangle$.

We are now in position to introduce the notion of a split square compatible with a span.

\begin{definition}\label{def1} A split square, such as $(\ref{split square})$, is said to be compatible with a span $(D,d,c)$ if  for every morphism \[u\colon{E\to D}\] with $du=due_2p_2$ and $cu=cue_1p_1$, there exists a unique morphism \[\theta\colon{A\times_B C\to D}\] such that $\theta\epsilon_1=ue_1$, $\theta\epsilon_2=ue_2$, $d\theta=due_2\pi_2$ and $c\theta=cue_1\pi_1$.
\end{definition}

When a split square is compatible with all the spans from a class of spans $\M$ then we say that it is $\M$-compatible.

\section{The main result}

We recall the functors that are involved in the statement of the main theorem as well as the ones that will be used in the proof.

\[
\xymatrix{
\Grpd(\C,\M)\ar[d]\ar[r]^{F_4^{\M}} & \RG(\C,\M)\ar@{=}[d]\\
\Cat(\C,\M)\ar[d]\ar[r]^{F_3^{\M}} & \RG(\C,\M)\ar@{=}[d]\\
\MG(\C,\M)\ar[r]^{F_2^{\M}} & \RG(\C,\M)\ar@<-.5ex>[d]\\
\PreGrpd(\C,\M)\ar@{-->}[u]\ar[d]\ar[r]^{F_1^{\M}} & \Span(\C,\M)\ar@<-.5ex>[d]\ar@{-->}@<-.5ex>[u] & \M\ar[l]_-{\cong}\\
\MKite(\C,\M)\ar[r]^{F_0^{\M}} & \Kite(\C,\M)\ar@<-.5ex>[u]_{}}
\]

\begin{theorem}\label{theorem 1}
Let $\C$ be a category with pullbacks and equalizers. If $\M$ is a class of spans in $\C$ which contains all identity spans and is stable under pullbacks, then t.f.c.a.e:
\begin{enumerate}
\item[(1)] $F_4^{\M}\colon{\Grpd(\C,\M)\to \RG(\C,\M)}$ has a section.
\item[(2)] $F_3^{\M}\colon{\Cat(\C,\M)\to \RG(\C,\M)}$ has a section.
\item[(3)] $F_2^{\M}\colon{\MG(\C,\M)\to \RG(\C,\M)}$ has a section.
\item[(4)] $F_1^{\M}\colon{\PreGrpd(\C,\M)\to \Span(\C,\M)}$ has a section.
\item[(5)] $F_0^{\M}\colon{\MKite(\C,\M)\to \Kite(\C,\M)}$ is an isomorphism.
\item[(6)] $F_1^{\M}\colon{\PreGrpd(\C,\M)\to \Span(\C,\M)}$ is an isomorphism.
\item[(7)] $F_2^{\M}\colon{\MG(\C,\M)\to \RG(\C,\M)}$ is an isomorphism.
\item[(8)] $F_3^{\M}\colon{\Cat(\C,\M)\to \RG(\C,\M)}$ is an isomorphism.
\item[(9)] $F_4^{\M}\colon{\Grpd(\C,\M)\to \RG(\C,\M)}$ is an isomorphism.
\item[(10)] Every split square in $\C$ is $\M$-compatible.
\end{enumerate}
\end{theorem}

\begin{proof}

 The proof is done by showing the following list of implications:
\begin{itemize}

\item[$(1)\Rightarrow(2)$] If the functor $F_4^{\M}$ has a section then by composition with the inclusion $\Grpd\to\Cat$ it gives the desired section to the functor $F_3^{\M}$.

\item[$(2)\Rightarrow(3)$] If the functor $F_3^{\M}$ has a section then by composition with the inclusion $\Cat\to\MG$ it gives the desired section to the functor $F_2^{\M}$.

\item[$(3)\Rightarrow(4)$] 
Start with a span in $\M$, take its image under the functor $K_1$ which is a reflexive graph (it is in $\M$, by pullback stability), and hence this reflexive graph is multiplicative and we can compose its multiplication with the morphism $d_2p_1=c_1p_2$ and this  gives a pregroupoid structure on the original span. See also \cite{NMF.VdL.14} where the same argument has been used.

\item[$(4)\Rightarrow(5)$] 

Let us assume that the functor $F_1^{\M}$ has a section. This means that every span $(D,d,c)$ in $\M$ is canonically equipped with a pregroupoid structure. First we will show that the functor $F_0^{\M}$ has a section and then we will prove that it is an isomorphism. To see that is has a section, consider any directed kite with its direction in $\M$, let us say
\begin{equation}\label{diag: kite6a}
\vcenter{\xymatrix@!0@=4em{A \ar@<.5ex>[r]^-{f} \ar[rd]_-{u} & B
\ar@<.5ex>[l]^-{r}
\ar@<-.5ex>[r]_-{s}
\ar[d]^-{w} & C \ar@<-.5ex>[l]_-{g} \ar[ld]^-{v}\\
& D\ar[dl]_{d}\ar[rd]^{p_1}\\D_0&&D_1}}
\end{equation}
then, by assumption and because its direction span $(D,d,c)$ is in $\M$, it is a pregroupoid. In other words, there is a natural pregroupoid structure $p_{d,c}\colon{D(d,c)\to D}$. The desired multiplication on the given directed kite is obtained by composing the morphism $p_{d,c}$ with the morphism $\theta\colon{A\times_B C\to D(d,c)}$, which is uniquely determined by $p_1\theta=\langle u, urf\rangle$ and $p_2\theta=\langle vsg, v\rangle$ (with $p_1$ and $p_2$ the morphisms obtained as in the kernel pair construction). It is readily checked that the needed conditions are satisfied with $m=p_{d,c}\theta$.

This shows that the functor $F_1^{\M}$ has a section. To prove that it is an isomorphism, we observe that the span $(A\times_B C, \pi_1,\pi_2)$ is a span in $\M$, because it is obtained by pullback from an identity span, and moreover it admits a section $$\delta\colon{A\times_B C\to A\times_B C(\pi_1,\pi_2)},$$ sending a pair $(a,c)$ to the zigzag 
\[\xymatrix{a&\ar[l] sf(a) \ar[r] & rg(c) & c \ar[l] }\]

If we denote by $p_{\pi_1,\pi_2}\colon{A\times_B C(\pi_1,\pi_2)\to A\times_B C}$ the morphisms asserting that the span $(A\times_B C, \pi_1,\pi_2)$ is difunctional, then we have $p_{\pi_1,\pi_2}\delta=1_{A\times_B C}$. If we now assume the existence of another multiplication, say $\varphi\colon{A\times_B C\to D}$ on our directed kite, then we observe that the naturality of $p$ gives rise to a commutative square
\[\xymatrix{
A\times_B C(\pi_1,\pi_2)\ar[d]_{\varphi^{3}}\ar[r]^-{p_{\pi_1,\pi_2}}& A\times_B C \ar[d]^{\varphi}\ar@{-->}[ld]_{\theta}\\
D(d,c)\ar[r]^{p_{d,c}} & D
}\]

from which we conclude $\varphi=p_{d,c}\theta$, since we have
\begin{eqnarray}
\varphi=\varphi p_{\pi_1,\pi_2}\delta=p_{d,c}\varphi^3 \delta=p_{d,c}\theta
\end{eqnarray}
This proves that $F_0^{\M}$ is an isomorphism.

\item[$(5)\Rightarrow(6)$] 

If $F_0^{\M}$ is an isomorphism then $F_1^{\M}$ is also an isomorphism. Example illustrating Diagram $(\ref{diag: kite5})$ tell us that every span gives rise to a directed kite. Moreover it is not difficult to see that a span is a pregroupoid if and only if the directed kite obtained by the kernel pair construction (see Diagram $(\ref{diag: kite5})$) is multiplicative. This shows that there is a unique pregroupoid structure on every span from $\M$ because it can be reduced to a special type of directed kite.

\item[$(5)\Rightarrow(7)$] Similarly to the previous item. Given a reflexive graph whose span part is in $\M$, we observe that it is a multiplicative graph if and only if the directed kite displayed as diagram $(\ref{diag: kite1})$ is a multiplicative kite. The naturality of the multiplication comes from the fact that the kite displayed in diagram $(\ref{diag: kite4})$ is multiplicative. Indeed, the uniqueness of its multiplicative structure forces $f_1m=mf_2$.

\item[$(5)\Rightarrow(8)$]

 Similarly to the previous item. Given a multiplicative graph whose span part is in $\M$, we observe that it is an internal category if and only if the directed kite displayed as Diagram $(\ref{diag: kite2})$ is a multiplicative kite. Indeed, the uniqueness of the multiplication forces the two possible morphisms $m(1\times m)$ and $m(m\times 1)$ to be equal because they are both candidates to a multiplicative structure on Diagram $(\ref{diag: kite2})$.

\item[$(5)\Rightarrow(9)$]
 Similarly to the previous item. Given an internal category whose span part is in $\M$, we observe that it is an internal groupoid if and only if the directed kite dislayed as Diagram $(\ref{diag: kite3})$ is a multiplicative kite. See \cite{NMF.08} for more details.

\item[$(5)\Rightarrow(10)$]
Consider a split square such as the one displayed in $(\ref{split square})$ together with a span in $\M$ and a morphism $u\colon{E\to D}$ with $du=due_2p_2$ and $cu=cue_1p_1$. Then, diagram $\ref{diag: kite6}$ combined with the morphism $u$ gives rise to a directed kite with the span $(D,d,c)$ as its direction
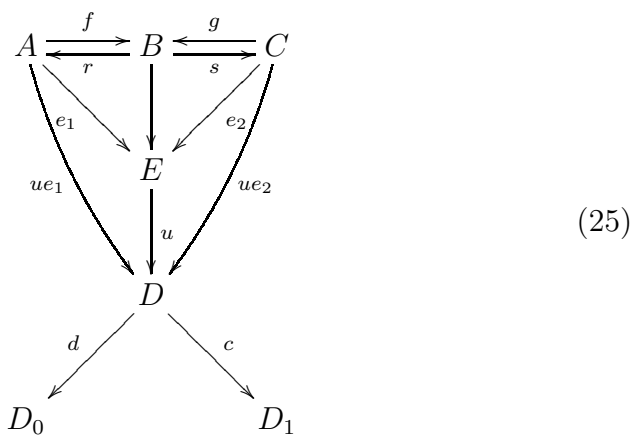
\begin{equation}\label{diag: kite7}
\vcenter{\xymatrix@!0@=4em{A \ar@<.5ex>[r]^-{f} \ar[rd]_-{e_1} \ar@/_/[ddr]_{ue_1} & B
\ar@<.5ex>[l]^-{r}
\ar@<-.5ex>[r]_-{s}
\ar[d]^-{} & C \ar@<-.5ex>[l]_-{g} \ar[ld]^-{e_2}\ar@/^/[ddl]^{ue_2}\\
& E\ar[d]^{u}\\
& D\ar[dl]_{d}\ar[rd]^{c}\\D_0&&D_1}}
\end{equation}
The multiplication for this kite gives the desired morphism $\theta\colon{A\times_B C\to D}$  satisfying the needed conditions to assert that the split square is compatible with the span.

\item[$(10)\Rightarrow(4)$]
Given a span $(D,d,c)$ we consider the split square
\[\xymatrix@!0@=4em{
E
\ar@<.5ex>[r]^-{p_2} \ar@<-.5ex>[d]_-{p_1}
& D(c)
\ar@<.5ex>[l]^-{e_2} \ar@<-.5ex>[d]_-{c_1}
 \\
D(d)
 \ar@<.5ex>[r]^-{d_2} \ar@<-.5ex>[u]_-{e_1}
& D
 \ar@<.5ex>[l]^-{\Delta} \ar@<-.5ex>[u]_-{\Delta}
 }
\]
which is based on the kernel pair construction, namely the objects $D(d)$ and $D(c)$, as well as the morphisms $d_2,c_2$ and $\Delta$. The remaining part is obtained by taking $E$ to be the equalizer of  the two morphisms
\[\xymatrix{E\ar[r] & D(d)(h)\ar@<.5ex>[r]^-{cd_1h_1}\ar@<-.5ex>[r]_-{cd_1h_2} & D_1}\]  
which are obtained by taking the kernel pair of $h=cd_2$,
\[\xymatrix{D(d)(h)\ar[r]^{h_2}\ar[d]_{h_1}&D(d)\ar[d]^{cd_2}\\D(d)\ar[r]^{cd2}&D_1}\]

In terms of sets and elements, that is when interpreting it in the category of sets and maps, if $x\in D$ is displayed as an arrow
\[\xymatrix{D_0\ni d(x)\ar[r]^{x} & c(x)\in D_1}\]
or simply pictured as $\xymatrix{\cdot\ar[r]^{x}&\cdot}$, then the elements in $D(d)$ are pictured as $\xymatrix{\cdot & \cdot\ar[l]_{x}\ar[r]^{y}&\cdot}$, we then have that the elements in $E$ are of the form (include reference here ??)
\[\xymatrix{\cdot & \cdot \ar[l]_{w}\ar[d]^{z}\\\cdot\ar[u]^{x}\ar[r]^{y}& \cdot}\]
while the maps $p_1,p_2,e_1,e_2$ are defined as follows
\begin{eqnarray}
p_1(x,y,z,w)=(x,y)\\
p_2(x,y,z,w)=(y,z)\\
e_1(x,y)=(x,y,y,x)\\
e_2(y,z)=(y,y,z,z)
\end{eqnarray}
\end{itemize}

By construction we also have a morphism $u\colon{E\to D}$, which in terms of elements would be defined as $u(x,y,z,w)=w$. The assumption that the split square which we have constructed is compatible with the given span, together with the existence of the morphism $u$, gives us a morphism $\theta\colon{D(d,c)\to D}$, which is precisely a pregroupoid structure on the span $(D,d,c)$.

The implications $(6)\Rightarrow(4)$, $(7)\Rightarrow(3)$, $(8)\Rightarrow(2)$ and $(9)\Rightarrow(1)$ are obvious. And this completes the proof.


\end{proof}

\section{Conclusion}

The result presented here suggests that every time a category can be characterized by one of the equivalent conditions in Theorem \ref{theorem 1}, for an appropriate class of spans, then it is a sort of Mal'tsev category. The three cases that we have considered, namely all spans, all relations, and all strong relations, characterise, respectively, naturally Mal'tsev categories, Mal'tsev categories and weakly Mal'sev categories (see, respectively, \cite{Johnstone},\cite{Bourn},\cite{NMF.VdL.14}). A new class of relations, namely positive relations, has recently been used \cite{Diana} to classify Goursat categories (a positive relation is of the form $U\circ U$, for some relation $U$). In a similar way it is expected that other examples will arise. Their study is postponed for future work.

%
%
%

\section*{Acknowledgements} 

The author wishes to warmly thank Zurab Janelidze and Tim Van der Linden for their kind collaboration on this project.

This work is supported by Funda\c c\~ ao para a Ci\^ encia e a Tecnologia (FCT) and Centro2020 through the Project references: UID/Multi/04044/2013; PAMI - ROTEIRO/0328/2013 (N. 022158); Next.parts (17963); and also by CDRSP and ESTG from the Polytechnic Institute of Leiria.



\begin{thebibliography}{99}

\bibitem{CLP} A. Carboni, J. Lambek and M. C. Pedicchio, \emph{Diagram chasing in Mal'cev categories}, J. Pure App. Algebra \textbf{69} (1990) 271-284.

\bibitem{BB} F. Borceux and D. Bourn, \emph{Mal'cev, Protomodular, Homological and Semi-Abelian
Categories}, Math. Appl. 566, Kluwer, 2004.

\bibitem{Bourn} D. Bourn, \emph{Mal'cev categories and fibration of pointed objects}, Appl. Categ. Struct. \textbf{4} (1996) 307--327.

\bibitem{Preord} M.M. Clementino, N. Martins-Ferreira and A. Montoli,  \emph{On the categorical behaviour of preordered groups},  DMUC pre-print 22 (2018)

\bibitem{Diana} M. Gran, D. Rodelo and I. Tchoffo-Nguefeu, \emph{ Variations of  the
Shifting Lemma and Goursat categories}, 	arXiv:1809.10408 [math.CT].

\bibitem{Johnstone} P. T. Johnstone, \emph{Affine categories and naturally Mal'cev categories}, J. Pure App. Algebra \textbf{61} (1989) 251-256.

\bibitem{Janelidze} G. Janelidze, \emph{Internal Crossed Modules}, Georgian Math. J. \textbf{10} (1) (2003) 99--114.

\bibitem{ZJ.NMF.12} Z. Janelidze and N. Martins-Ferreira, \emph{Weakly Mal'tsev categories and strong relations}, Theory Appl. Categ.  \textbf{27} (5) (2012) 65--79.

\bibitem{AndersKock} A. Kock, \emph{Fibre Bundles in General Categories},  J. Pure Appl. Algebra  \textbf{56}  (1989)  233--245. 

\bibitem{Mal'tsev} A. I. Mal'tsev, \emph{Algebraic Systems}, Springer-Verlag, New York, 1973.


\bibitem{NMF.08} N. Martins-Ferreira, \emph{Weakly Mal'cev categories}, Theory Appl. Categ.  \textbf{21} (6) (2008)  91--117.


\bibitem{NMF.12} N. Martins-Ferreira, \emph{On distributive lattices and weakly Mal'tsev categories}, J. Pure Appl. Algebra \textbf{216} ( 8-9) (2012)  1961--1963 .

\bibitem{NMF.17} N. Martins-Ferreira, \emph{Normalized Bicategories Internal to Groups and more General Mal'tsev Categories}, Appl. Categ. Struct. \textbf{25} (6) (2017) 1137--1158.


\bibitem{NMF.VdL.14} N. Martins-Ferreira and T. Van der Linden, \emph{Categories vs. groupoids via generalised Mal'tsev properties}, Cah. Topol. Géom. Différ. Catég. \textbf{55} (2) (2014)  83--112 .

\bibitem{Pedicchio} M.C. Pedicchio, \emph{Maltsev categories and Maltsev operations}, Journal of Pure and Applied Algebra \textbf{98} (1)
 (1995) 67--71



\end{thebibliography}

\end{document}